\documentclass[11pt]{amsart}
\usepackage{amsmath,amsthm,amsfonts,amssymb,latexsym,epsfig}

\newtheorem{theorem}{Theorem}[section]

\newtheorem{lemma}{Lemma}[section]

\newtheorem{cor}{Corollary}[section]

\numberwithin{equation}{section}

\theoremstyle{definition}

\theoremstyle{remark}

\begin{document}
\title{Finite Sections of Weighted Carleman's Inequality}
\author{Peng Gao}
\address{Department of Computer and Mathematical Sciences,
University of Toronto at Scarborough, 1265 Military Trail, Toronto
Ontario, Canada M1C 1A4} \email{penggao@utsc.utoronto.ca}
\date{June 30, 2007.}


\begin{abstract}
 We study finite sections of weighted Carleman's inequality following the approach of De Bruijn. Similar to the unweighted case, we obtain an asymptotic expression for the optimal constant.
\end{abstract}

\maketitle
\section{Introduction}
\label{sec 1} \setcounter{equation}{0}
  
   The well-known Carleman's inequality asserts that for convergent infinite series $\sum a_n$ with non-negative terms, one has
\begin{equation*}
   \sum^\infty_{n=1}(\prod^n_{k=1}a_k)^{\frac 1{n}}
\leq e\sum^\infty_{n=1}a_n,
\end{equation*}
   with the constant $e$ best possible.

  There is a rich literature on many different proofs of Carleman's inequality as well as its generalizations and extensions. We shall refer the readers to the survey articles  \cite{P&S} and \cite{D&M} as well as the references therein for
an account of Carleman's inequality. 

  From now on we will assume $a_n \geq 0$ for $n \geq 1$ and any
   infinite sum converges. In \cite{G2}, the author studied the following weighted Carleman's inequality:
\begin{equation}
\label{1}
   \sum^\infty_{n=1}G_n
\leq U\sum^\infty_{n=1}a_n,
\end{equation}
  where
\begin{equation*}
   G_n=\prod^n_{k=1}a^{\lambda_k/\Lambda_n}_k, \hspace{0.1in} \Lambda_n=\sum^n_{k=1}\lambda_k, ~~\lambda_k \geq 0, ~~\lambda_1>0.
\end{equation*}
   Using Carleman's original approach in \cite{Carlman}, the author \cite{G2} proved the following:
\begin{theorem}
\label{thm1}
  Suppose that
\begin{equation}
\label{5}
  M=\sup_n\frac
    {\Lambda_n}{\lambda_n}\log \Big(\frac {\Lambda_{n+1}/\lambda_{n+1}}{\Lambda_n/\lambda_n} \Big ) < +\infty,
\end{equation}
  then inequality \eqref{1} holds with $U=e^M$.
\end{theorem}
 
   In this paper, we consider finite sections of weighted Carleman's inequality \eqref{1}:
\begin{equation}
\label{3}
   \sum^N_{n=1}G_n
\leq \mu_N\sum^N_{n=1}a_n.
\end{equation}   
   where $N \geq 1$ is any integer. In the case of $\lambda_k=1$ (the unweighted case), De Bruijn \cite{De} had shown that the best constant satisfies
\begin{equation*}
   \mu_N=e-\frac {2\pi^2e}{(\log N)^2}+ O \Big (\frac 1{(\log N)^3} \Big ).
\end{equation*}  

   It is our goal in this paper to obtain similar asymptotic expressions for $\mu_N$ for the weighted Carleman's inequality following De Bruijn's approach in \cite{De}. We shall prove the following
\begin{theorem}
\label{thm2}
  Assume \eqref{5} holds with $\{ \lambda_k \}^{\infty}_{k=1}$ a non-decreasing sequence satisfying 
\begin{eqnarray}
\label{1.3}
  \sup_k \frac {\lambda_{k+1}}{\lambda_k} &<& +\infty, \\
\label{1.4}
    M+\log (\lambda_k/\lambda_{k+1}) &\leq & \frac {\Lambda_{k+1}}{\lambda_k}\log \Big(\frac {\Lambda_{k+1}/\lambda_{k+1}}{\Lambda_k/\lambda_k} \Big ), \\
\label{1.5}
    \frac
    {\Lambda_k}{\lambda_k}\log \Big(\frac {\Lambda_{k+1}/\lambda_{k+1}}{\Lambda_k/\lambda_k} \Big ) &=& M+O\Big(\frac {\lambda_k}{\Lambda_k}  \Big ), \\
\label{3.5'}
  \frac {\lambda_k}{\Lambda_k} &=& \frac {C}{k}+O(\frac 1{k^2}), \hspace{0.1in} C>0, \\
\label{1.6}
  \inf_k\Big (\frac {\Lambda_{k+1}}{\lambda_{k+1}}-\frac {\Lambda_k}{\lambda_k} \Big ) &>& 0.
\end{eqnarray}
  Then for any integer $N \geq 1$, inequality \eqref{3} holds with the best constant satisfying:
\begin{equation*}
   \mu_N=e^M- \frac {2\pi^2e^M}{C^2(\log N)^2}+O\Big (\frac 1{(\log N)^3} \Big ).
\end{equation*}  
\end{theorem}

   We note here that \eqref{1.6} implies $M>0$, which we shall use without further mentioning throughout the paper. We may also assume $N \geq 2$ from now on.
\section{Preliminary Treatment}
\label{sec 2} \setcounter{equation}{0}
   It is our goal in this section to give an upper bound for the number $U_N$ appearing in \eqref{3}.  We first recall the author's approach in \cite{G2} (following that of Carleman in \cite{Carlman}) for determining the maximum value $\mu_N$ of $\sum^N_{n=1}G_n$ in \eqref{3} subject to the constraint $\sum^N_{n=1}a_n=1$ using Lagrange multipliers. It is easy to see that we may assume $a_n > 0$ for all $1 \leq n \leq N$ when the maximum is reached. We now define
\begin{equation*}
  F({\bf a}; \mu)=\sum^N_{n=1}G_n-\mu (\sum^N_{n=1}a_n-1),
\end{equation*}
  where ${\bf a}=(a_n)_{1 \leq n \leq N}$. By the Lagrange method, we have to solve $\nabla F=0$, or the following system of equations:
\begin{equation}
\label{2.1}
  \mu a_k=\sum^N_{n=k}\frac {\lambda_kG_n}{\Lambda_n}, \hspace{0.1in} 1 \leq k \leq N; \hspace{0.1in} \sum^N_{n=1}a_n=1.
\end{equation}
   We note that on summing over $1 \leq k \leq N$ of the first $N$ equations above, we get
\begin{equation*}
  \sum^N_{n=1}G_n=\mu.
\end{equation*}
   Hence we have $\mu=\mu_N$ in this case which allows us to recast the equations \eqref{2.1} as:
\begin{equation*}
  \mu_N \frac {a_k}{\lambda_k}=\sum^N_{n=k}\frac {G_n}{\Lambda_n}, \hspace{0.1in} 1 \leq k \leq N; \hspace{0.1in} \sum^N_{n=1}a_n=1.
\end{equation*}
  On subtracting consecutive equations, we can rewrite the above system of equations as:
\begin{equation*}
  \mu_N (\frac {a_k}{\lambda_k}-\frac {a_{k+1}}{\lambda_{k+1}})=\frac {G_k}{\Lambda_k}, \hspace{0.1in} 1 \leq k \leq N-1; \hspace{0.1in}  \mu_N \frac {a_N}{\lambda_N}=\frac {G_N}{\Lambda_N}; \hspace{0.1in} \sum^N_{n=1}a_n=1.
\end{equation*}
   
   Now following the notations in \cite{De}, we define  for $1 \leq k \leq N-1$ (this is different from the treatment in \cite{G2}),
\begin{equation*}
  h_k = \log \frac {G_k}{a_k},
\end{equation*}
  so that we can obtain a recursion expressing $h_{k+1}$ in terms of $h_k$ as follows:
\begin{equation*}
  h_{k+1}=\frac {\Lambda_k}{\Lambda_{k+1}}h_k-\frac {\Lambda_k}{\Lambda_{k+1}}\log \Big (\frac {\lambda_{k+1}}{\lambda_{k}}- \frac {\lambda_{k+1}}{\Lambda_{k}\mu_N}e^{h_k} \Big ),  \hspace{0.1in} 1 \leq k \leq N-1.
\end{equation*}
   We now define a sequence of real functions $h_k(\mu)$ inductively by setting $h_1(\mu)=0$ and
\begin{equation}
\label{2.2}
   h_{k+1}(\mu)=\frac {\Lambda_k}{\Lambda_{k+1}}h_k(\mu)-\frac {\Lambda_k}{\Lambda_{k+1}}\log \Big (\frac {\lambda_{k+1}}{\lambda_{k}}- \frac {\lambda_{k+1}}{\Lambda_{k}\mu}e^{h_{k}(\mu)} \Big ),  \hspace{0.1in} 1 \leq k \leq N-1.
\end{equation}
   We note that $h_k(\mu_N)= h_k$ for $1 \leq k \leq N-1$ and 
\begin{eqnarray*}
   h_N(\mu_N) &=& \frac {\Lambda_{N-1}}{\Lambda_{N}}h_{N-1}(\mu_N)-\frac {\Lambda_{N-1}}{\Lambda_{N}}\log \Big (\frac {\lambda_{N}}{\lambda_{N-1}}- \frac {\lambda_{N}}{\Lambda_{N-1}\mu_N}e^{h_{N-1}(\mu_N)} \Big ) \\
 &=& \frac {\Lambda_{N-1}}{\Lambda_{N}}\log \Big (\frac {G_{N-1}}{a_{N-1}}\Big ) \\
 && -\frac {\Lambda_{N-1}}{\Lambda_{N}}\log \Big (\frac {\lambda_{N}}{\lambda_{N-1}}- \frac {\lambda_{N}}{\Lambda_{N-1}\mu_N}\Big (\mu_N \Big(\frac {\Lambda_{N-1}}{\lambda_{N-1}}-\frac {\Lambda_{N-1}}{\lambda_{N}}\frac {a_{N}}{a_{N-1}}\Big)\Big ) \Big ) \\
 &=& \frac {\Lambda_{N-1}}{\Lambda_{N}}\log \Big (\frac {G_{N-1}}{a_{N}}\Big )
  =\log \Big (\frac {G_{N}}{a_{N}}\Big )=\log (\frac {\mu_N\Lambda_N}{\lambda_N}).
\end{eqnarray*}
   
   We now show by induction that if $\mu \geq e^M$, then for any $2 \leq k \leq N$,
\begin{equation}
\label{2.3}
   h_k(\mu) \leq M\frac {\Lambda_{k-1}}{\Lambda_k}.
\end{equation}
   As we have seen above that $h_N(\mu_N)=\log (\mu_N\Lambda_N/\lambda_N)  \geq \log \mu_N \geq M$ when $\mu_n \geq e^M$, this forces $\mu_N < e^M$.

   Now, to establish \eqref{2.3}, we first consider the case $k=2$. As $h_1=0$, We have by \eqref{2.2},
\begin{equation}
\label{2.4}
  h_2(\mu)=-\frac {\Lambda_1}{\Lambda_{2}}\log \Big (\frac {\lambda_{2}}{\lambda_{1}}- \frac {\lambda_{2}}{\Lambda_{1}\mu} \Big ).
\end{equation}
  It is easy to see that $h_1(\mu) \leq M\Lambda_1/\Lambda_2$ is equivalent to
\begin{equation*}
  \frac {\lambda_2}{\lambda_{1}}e^M \geq \frac {\lambda_2}{\Lambda_1}\frac {e^M}{\mu}+1.
\end{equation*}
  As $e^M/\mu \leq 1$, the above inequality follows easily from the assumption \eqref{5}. Now assume inequality \eqref{2.3} holds for $k \geq 2$, then by \eqref{2.2} again, it is easy to see that for \eqref{2.3} to hold for $k+1$, it suffices to show that
\begin{equation*}
  \frac {\lambda_{k+1}}{\lambda_{k}}e^{M\lambda_k/\Lambda_k} \geq \frac {\lambda_{k+1}}{\Lambda_k}\frac {e^M}{\mu}+1,
\end{equation*}
   and this again follows easily from the assumption \eqref{5}.

\section{The Breakdown Index}
\label{sec 3} \setcounter{equation}{0}
    As in \cite{De}, we now try to evaluate $h_k(\mu)$ consecutively from \eqref{2.2} for any $\mu>0$, starting with $h_1=0$. Certainly we are only interested in the real values of $h_k$ and hence we say that the procedure breaks down at the first $k$ where $\lambda_{k+1}/\lambda_k-\lambda_{k+1}/(\Lambda_{k}\mu)e^{h_{k}(\mu)} \leq 0$, or equivalently, 
\begin{equation}
\label{3.1}
 h_k(\mu) \geq \log (\mu \Lambda_k/\lambda_k).
\end{equation} 
   We define the breakdown index $N_{\mu}$ as the smallest $k$ for which inequality \eqref{3.1} holds if there is such a $k$ and we put $N_{\mu} = +\infty$ otherwise. Thus for all $\mu>0$ we can say that $h_k(\mu)$ is defined for all $k \leq N_{\mu}$.

    Note that \eqref{2.3} implies $N_{\mu} = +\infty$ when $\mu \geq e^M$.  So from now on we may assume $0 < \mu < e^M$ and it is convenient to have some monotonicity properties available in this case. We have $h_1(\mu)=0$ for $0 < \mu < e^M$ and we let $\mu_1$ be the largest $\mu$ for which inequality \eqref{3.1} holds for $k=1$, this implies $\mu_1=1$. Now $h_2(\mu)$ is defined for $\mu > \mu_1$, and $h_2(\mu)$ is given by \eqref{2.4}, which is a decreasing function of $\mu$ for $\mu > \mu_1$. Note also that the right-hand side expression of inequality \eqref{3.1} is an increasing function of $\mu$ for any fixed $k$. It follows that
\begin{equation*}
  \lim_{\mu \rightarrow \mu_1^+}h_2(\mu)=+\infty; \hspace{0.1in} h_2(e^M) \leq M(1-\lambda_{2}/\Lambda_2) < \log (e^M\Lambda_2/\lambda_2) \leq \log (\mu \Lambda_2/\lambda_2).
\end{equation*}
   Thus there is exactly one value of $\mu < e^M$ for which inequality \eqref{3.1} holds with equality for $k=2$ and we define this value of $\mu$ to be $\mu_2$. This procedure can be continued. At each step we argue that $h_k(\mu)$ is defined and decreasing for $\mu > \mu_{k-1}$, that 
\begin{equation*}
  \lim_{\mu \rightarrow \mu_{k-1}^+}h_k(\mu)=+\infty; \hspace{0.1in} h_k(e^M) \leq M(1-\lambda_{k}/\Lambda_k) < \log (e^M\Lambda_k/\lambda_k) \leq \log (\mu \Lambda_k/\lambda_k).
\end{equation*}
   We then infer that $\mu_k$ is uniquely determined by $h_k(\mu) = \log (\mu \Lambda_k/\lambda_k)$. Moreover, $h_{k+1}(\mu)$ is again defined and decreasing for $\mu > \mu_{k}$ as both terms on the right of \eqref{2.2} are decreasing functions of $\mu$.
  
   Thus by induction we obtain that
\begin{equation}
\label{3.1'}
   1=\mu_1 < \mu_2 < \mu_3 < \ldots < e^M,
\end{equation}
   and that $h_{k+1}(\mu)$ is defined and decreasing for $\mu > \mu_{k}$. Moreover, $h_k(\mu) >  \log (\mu \Lambda_k/\lambda_k)$ if $\mu_{k-1} < \mu < \mu_k$, 
$h_k(\mu_k)=\log (\mu_k \Lambda_k/\lambda_k)$, $h_k(\mu) <  \log (\mu_k \Lambda_k/\lambda_k)$ if $ \mu > \mu_k$.

  It follows that the breakdown index $N_{\mu}$ equals $1$ if $\mu \leq \mu_1$, $2$ if $\mu_1 < \mu \leq \mu_2$, etc. We remark here that for fixed $\mu \leq e^M$, the $h_k(\mu)$'s are non-negative and increase as $k$ increases from $1$ to $N_k$. This follows from \eqref{2.2} by noting that
\begin{equation}
\label{3.2}
  h_{k+1}(\mu)-h_k(\mu)=-\frac {\lambda_k}{\Lambda_{k+1}}h_k(\mu)-\frac {\Lambda_k}{\Lambda_{k+1}}\log \Big (\frac {\lambda_{k+1}}{\lambda_{k}}- \frac {\lambda_{k+1}}{\Lambda_{k}\mu}e^{h_{k}(\mu)} \Big ).
\end{equation}
   It thus suffices to show the right-hand side expression above is non-negative. Equivalently, this is $\lambda_{k+1}/\lambda_{k} \leq f(h_k(\mu))$, where
\begin{equation*}
 f(x)= \frac {\lambda_{k+1}}{\Lambda_{k}\mu}e^{x}+ e^{-\frac {\lambda_k}{\Lambda_{k}}x} .
\end{equation*}
  It is easy to see that $f(x)$ is minimized at $x_0=\Lambda_k/\Lambda_{k+1}\log(\lambda_k\mu/\lambda_{k+1})$. Note also that
\begin{equation*}
 \frac {\lambda_{k+1}}{\Lambda_{k}\mu}e^{x_0}=\frac {\lambda_k}{\Lambda_{k}} e^{-\frac {\lambda_k}{\Lambda_{k}}x_0}.
\end{equation*}
   It follows that
\begin{equation*}
 f(x) \geq f(x_0)=\frac {\Lambda_{k+1}}{\Lambda_{k}} e^{-\frac {\lambda_k}{\Lambda_{k}}x_0}.
\end{equation*}
   It follows from \eqref{1.4} that
\begin{equation*}
    \log (\lambda_ke^M/\lambda_{k+1})=M+\log (\lambda_k/\lambda_{k+1}) \leq \frac {\Lambda_{k+1}}{\lambda_k}\log \Big(\frac {\Lambda_{k+1}/\lambda_{k+1}}{\Lambda_k/\lambda_k} \Big ).
\end{equation*}
  It is easy to see that the above inequality implies that $f(x_0) \geq \lambda_{k+1}/\lambda_{k}$ so that the $h_k(\mu)$'s increase as $k$ increases from $1$ to $N_k$. 

  The breakdown condition \eqref{3.1} is slightly awkward. We now replace it by a simpler one, for example, $h_k > \max (2,2M)$, by virtue of the following argument. Let $0 < \mu < e^M$ and assume that $N$ is such that $h_N > \max (2, 2M)$. Note that \eqref{3.5'} implies that $\lim_{k \rightarrow +\infty}\Lambda_k/\lambda_k = +\infty$ so that the right-hand side expression of \eqref{3.1} approaches $+\infty$ as $k$ tends to $+\infty$. Hence we may assume $N_{\mu} \geq N$ without loss of generality. Then we have
\begin{equation*}
  \log N_{\mu}- \log N =O(1).
\end{equation*}
   For, if $N \leq k \leq N_{\mu}$, the right-hand side of \eqref{3.2} equals
\begin{eqnarray}
 && -\frac {\lambda_k}{\Lambda_{k+1}}h_k(\mu)-\frac {\Lambda_k}{\Lambda_{k+1}}\log \Big (\frac {\lambda_{k+1}}{\lambda_{k}} \Big) -\frac {\Lambda_k}{\Lambda_{k+1}} \log \Big (1-\frac {\lambda_{k}}{\Lambda_{k}\mu}e^{h_{k}(\mu)} \Big )  \nonumber \\
\label{4.0}
&=& -\frac {\Lambda_k}{\Lambda_{k+1}}\log \Big (\frac {\lambda_{k+1}}{\lambda_{k}} \Big) 
+\frac {\lambda_k}{\Lambda_{k+1}}\Big (\frac {e^{h_{k}(\mu)}}{\mu}-h_k(\mu) \Big )+ \frac {\Lambda_{k}}{\Lambda_{k+1}}\sum^{+\infty}_{i=2}\frac 1{i}\Big ( \frac {\lambda_{k}}{\Lambda_{k}}\frac {e^{h_{k}(\mu)}}{\mu} \Big )^i   \\
& \geq & \frac {\lambda_k}{\Lambda_{k+1}}\Big (\frac {e^{h_{k}(\mu)}}{\mu}-h_k(\mu)-\frac {\Lambda_k}{\lambda_{k}}\log \Big (\frac {\lambda_{k+1}}{\lambda_{k}} \Big)  \Big ). \nonumber
\end{eqnarray}
   Note that, in view of \eqref{5} and \eqref{1.5},
\begin{equation}
\label{4.10}
   \frac {\Lambda_k}{\lambda_{k}}\log \Big (\frac {\lambda_{k+1}}{\lambda_{k}} \Big) =-\frac
    {\Lambda_k}{\lambda_k}\log \Big(\frac {\Lambda_{k+1}/\lambda_{k+1}}{\Lambda_k/\lambda_k} \Big )+\frac {\Lambda_k}{\lambda_{k}}\log \Big (\frac {\Lambda_{k+1}}{\Lambda_{k}} \Big) =1-M+O\Big(\frac {\lambda_k}{\Lambda_k}  \Big ).
\end{equation}
  As $(e^{h-M}-h+M-1)h^{-2}$ increases for $h \geq \max (2, 2M)$, we conclude that there exists a constant $C_0>0$ and an integer $N_0$ independent of $\mu$ such that for $k \geq N_0$,
\begin{eqnarray*}
  h_{k+1}(\mu)-h_k(\mu) &\geq & \frac {\lambda_{k}}{\Lambda_{k+1}}\Big(e^{h_{k}(\mu)-M}-h_k(\mu)+M-1 \Big)+O\Big(\frac {\lambda^2_k}{\Lambda_k\Lambda_{k+1}}  \Big ) \\
& >& \frac {C_0\lambda_{k}}{\Lambda_{k+1}}h^2_k(\mu).
\end{eqnarray*}
  We may assume $N \geq N_0$ from now on without loss of generality and we now simply the above relations by defining $d_N, d_{N+1}, \ldots$, starting with $d_N=h_N$, and
\begin{equation}
\label{3.3}
  d_{k+1}-d_k= \frac {C_0\lambda_{k}}{\Lambda_{k+1}}d^2_k.
\end{equation}
  
   Obviously we have $d_k \leq h_k \leq \log (\mu \Lambda_k/\lambda_k)$ for $N \leq k \leq N_{\mu}$. We use the bound
\begin{equation*}
   \log (\mu \Lambda_k/\lambda_k) \leq M+\log (\Lambda_k/\lambda_k) \leq M-1+\Lambda_k/\lambda_k \leq M\Lambda_k/\lambda_k.
\end{equation*}
  to get that $d_k \leq M\Lambda_k/\lambda_k$ for $N \leq k \leq N_{\mu}$. It follows from \eqref{3.3} that 
\begin{equation*}
  d_{k+1}-d_k \leq C_0Md_k.
\end{equation*}
   The above implies that we have $d_{k+1} \leq (C_0M+1)d_k$ for $N \leq k \leq N_{\mu}$ and \eqref{3.3} further implies that
\begin{equation}
\label{3.5}
   d_{k+1}-d_k \geq \frac {C_0\lambda_{k}}{(C_0M+1)\Lambda_{k+1}}d_kd_{k+1}.
\end{equation}
   We now apply \eqref{3.5'} to obtain via \eqref{3.5} that there exists a constant $C_1>0$ and an integer $N_1$ independent of $\mu$ such that for $k \geq N_1$, 
\begin{equation*}
   d^{-1}_{k}-d^{-1}_{k+1} \geq \frac {C_1}{k+1}.
\end{equation*}
   Certainly we may assume $N \geq N_1$ as well. Summing the above for $N \leq k \leq N_{\mu}-1$ yields:
\begin{equation*}
   \frac {1}{\max (2, 2M)} \geq d^{-1}_{N} \geq \sum_{N \leq k \leq N_{\mu}-1}\frac {C_1}{k+1}.
\end{equation*}
   It follows from this that
\begin{equation}
\label{3.6}
  \log N_{\mu}- \log N=-\sum_{N \leq k \leq N_{\mu}-1}\log (\frac {k}{k+1})\leq \sum_{N \leq k \leq N_{\mu}-1}\frac {1}{k+1}+O(1)=O(1).
\end{equation}

   We shall see in what follows that the relation \eqref{3.6} implies that there is no harm studying $\log N$ in stead of $\log N_{\mu}$. So from now on we shall concentrate on finding the smallest $k$ such that $h_k(\mu) > \max (2, 2M)$.
\section{Heuristic Treatment}
\label{sec 4} \setcounter{equation}{0}
   Our problem is, roughly, to determine how many steps we have to take in our recurrence \eqref{3.2} in order to push $h_k$ beyond the value of $\max (2, 2M)$,  assuming that $\mu$ is fixed, $\mu < e^M$ and $\mu$ close to $e^M$. 
Now assume we are able to neglect all the other terms of the right-hand side expression in \eqref{4.0} other than the first two terms, then we have a recurrence which can be written as
\begin{equation*}
  \Delta h=\frac {\lambda_k}{\Lambda_{k+1}}\Big (\frac {e^{h_{k}(\mu)}}{\mu}-h_k(\mu)-\frac {\Lambda_k}{\lambda_{k}}\log \Big (\frac {\lambda_{k+1}}{\lambda_{k}} \Big)  \Big ).
\end{equation*}
  In view of \eqref{4.10}, we may replace the last term above by $1-M$ and we may further consider the following recurrence using \eqref{3.5'}:
\begin{equation*}
  \Delta h=\frac {C}{k+1}\Big (\frac {e^{h_{k}(\mu)}}{\mu}-h_k(\mu)+M-1 \Big ).
\end{equation*}
   Next we consider $k$ as a continuous variable, and we replace the above by the corresponding differential equation, that is, we replace $\Delta h$ by $dh/dk$. Then we get
\begin{equation*}
  \frac {d \log (k+1)}{dh}=C^{-1}\Big (\mu^{-1}e^{h}-h+M-1 \Big )^{-1}.
\end{equation*}
   This suggests that if $N$ is the number of steps necessary to increase $h$ from $0$ to about $\max (2, 2M)$, then $\log N$ is roughly equal to
\begin{equation}
\label{4.1}
   \frac 1{C}\int^{\max (2, 2M)}_0\frac {dh}{\mu^{-1}e^{h}-h+M-1}.
\end{equation}
   The integrand has its maximum at $h=\log \mu$, and this is close to $M$. In the neighborhood of that maximum it can be approximated by
\begin{equation*}
  \frac 1{2}(h-\log \mu)^2+M-\log \mu.
\end{equation*}
   Therefore the value of \eqref{4.1} can be compared with
\begin{equation*}
   \frac 1{C}\int^{+\infty}_{-\infty}\frac {dh}{ \frac 1{2}(h-\log \mu)^2+M-\log \mu}=\frac {\sqrt{2}\pi}{C}\Big (\log (e^M/\mu) \Big )^{-1/2}.
\end{equation*}
   From this we see that for $\mu < e^M, \mu \rightarrow e^M$, we expect to have
\begin{equation}
\label{4.2}
  \log N_{\mu}=\frac {\sqrt{2}\pi}{C}\Big (\log (e^M/\mu ) \Big )^{-1/2}+O(1).
\end{equation}
   From this we see that if $ \mu \rightarrow e^M$, then $\log N_{\mu}$ tends to infinity. This also implies that for the sequence $\{ \mu_k \}$ defined as in \eqref{3.1'}, one must have $\lim_{k \rightarrow +\infty}\mu_k = e^M$. For otherwise, the sequence $\{ \mu_k \}$ is bounded above by a constant $<e^M$ and on taking any $\mu$ greater than this constant (and less than $e^M$), then the left-hand side of \eqref{4.2} becomes infinity (by our definition of $N_{\mu}$) but the right-hand side of \eqref{4.2} stays bounded, a contradiction. 

   Note that if $\mu=\mu_N$, then $N_{\mu}=N$, it follows from \eqref{4.2} that
\begin{equation*}
  \log (e^M/\mu_N)= \frac {2\pi^2}{C^2}\Big ( \log N+O(1) \Big )^{-2}.
\end{equation*}
   It is easy to see that the above leads to the following asymptotic expression for $\mu_N$:
\begin{equation*}
  \mu_N =e^M- \frac {2\pi^2e^M}{C^2(\log N)^2}+O\Big (\frac 1{(\log N)^3} \Big ).
\end{equation*}

   There are various doubtful steps in our argument above, but the only one that presents a serious difficulty is the omitting of all the other terms of the right-hand side expression of \eqref{4.0}. Certainly those terms can be expected to give only a small contribution if $k$ is large but the question is whether this contribution is small compared to $\mu^{-1}e^{h}-h+M-1$. The latter expression can be small if both $h_k-M$ and $\mu-e^M$ are small, and it is especially in that region that the integrand of \eqref{4.1} produces its maximal effect.

\section{Lemmas}
\label{sec 5} \setcounter{equation}{0}
\begin{lemma}
\label{lem5.1}
  For any given number $\eta>0, 0< \epsilon <M$, one can find an integer $k_0 >\eta$ and a number $\beta$, $e^{M-1}<\beta<e^M$ such that for $\beta < \mu \leq e^M$, 
\begin{equation}
\label{5.00}
 M-\epsilon < h_{k_0}(\mu) < \log \mu-\frac {M}{2} \frac {\lambda_{k}}{\Lambda_k}. 
\end{equation} 
\end{lemma}
\begin{proof}
  Note first that by \eqref{2.3} and our discussions in Section \ref{sec 3} that the $h_k(e^M)$'s are non-negative, we have 
\begin{equation*}
  0 \leq h_k(e^M) \leq M\frac {\Lambda_{k-1}}{\Lambda_k}.
\end{equation*}
  Let $k_1$ be an integer so that for all $k \geq k_0$,
\begin{equation*}
   M\frac {\Lambda_{k-1}}{\Lambda_k} > M-\epsilon.
\end{equation*}
   We may assume that $k \geq k_1$ from now on and note that not all $h_k(e^M)$ are $ \leq M-\epsilon$. Otherwise, it follows from \eqref{3.2}, \eqref{3.5'}, \eqref{4.0} and  \eqref{4.10} that
\begin{equation*}
  h_{k+1}(e^M)-h_k(e^M) \geq \frac {\lambda_k}{\Lambda_{k+1}}\Big (\frac {e^{h_{k}(e^M)}}{e^M}-h_k(e^M)+M-1 \Big )+O(\frac {1}{k^2}).
\end{equation*}
  Note that if $h_k(e^M) \leq M-\epsilon$ then
\begin{equation*}
  \frac {e^{h_{k}(e^M)}}{e^M}-h_k(e^M)+M-1 \geq e^{M-\epsilon-M}-M+\epsilon+M-1>0.
\end{equation*}
  It follows from \eqref{3.5'} and the fact that $\sum^{\infty}_{k=k_1}(k+1)^{-1}=+\infty$ that this leads to a contradiction. Thus there is an integer $k_0 > \eta$ for which
\begin{equation*}
   M-\epsilon < h_{k_0}(e^M) \leq M \frac {\Lambda_{k-1}}{\Lambda_k}< \log e^M-\frac {M}{2} \frac {\lambda_{k}}{\Lambda_k}. 
\end{equation*}
   Having fixed $k_0$ this way, we remark that $h_{k_0}(\mu)$ is continuous at $\mu=e^M$ and the lemma follows.
\end{proof}

\begin{lemma}
\label{lem5.2}
  There exist numbers $\beta$, $e^{M-1}<\beta<e^M$, and $c>0$, $0<\delta <1$ such that for all $\mu$ satisfying $\beta < \mu \leq e^M$, and for all $k$ satisfying $1 \leq k \leq N_{\mu}$ ($N_{\mu}$ is the breakdown index) we have  
\begin{equation}
\label{5.0}
 \frac {e^{h_{k}(\mu)}}{\mu}-h_k(\mu)+M-1 > c \Big ( \frac {\Lambda_k}{\lambda_k} \Big )^{-\delta}. 
\end{equation} 
\end{lemma}
\begin{proof}
  We apply Lemma \ref{lem5.1} with $\eta$ large enough so that the following inequality holds for any integer $k \geq \eta$:
\begin{equation}
\label{5.1}
\frac {\lambda_{k}}{\Lambda_{k+1}}\Big | -\frac {\Lambda_k}{\lambda_{k}}\log \Big (\frac {\lambda_{k+1}}{\lambda_{k}} \Big) +1-M \Big | + \frac {\Lambda_{k}}{\Lambda_{k+1}}\sum^{+\infty}_{i=2}\frac 1{i}\Big ( \frac {\lambda_{k}}{\Lambda_{k}} \Big )^i < \frac 3{4}\frac {\lambda^2_{k}}{\Lambda_{k}\Lambda_{k+1}}. 
\end{equation}  
  We shall also choose $\epsilon$ small enough so that we obtain values of $k_0$ and $\beta$. Without loss of generality, we may assume $\mu < e^M$ and for the time being we keep $\mu$ fixed ($\beta < \mu < e^M$) and we write $h_k$ instead of $h_k(\mu)$.

  As we remarked in Section \ref{sec 3}, the sequence $h_{k_0}, h_{k_0+1}, \ldots$ is increasing, possibly until breakdown. We shall now first consider those integers $k \geq k_0$ for which $h_{k}< \log \mu$. For those $k$ we can prove
\begin{equation}
\label{5.2}
  h_{k+1}-h_k < \frac {\lambda_k}{\Lambda_{k+1}}\Big ( \frac 1{2}(\log \mu-h_{k})^2+\log (\frac {e^M}{\mu})+ \frac 3{4} \frac {\lambda_k}{\Lambda_k}  \Big ).
\end{equation}
   This follows by \eqref{3.2} and \eqref{4.0}, using $e^{-u}<1-u+u^2/2$, where $u=\log \mu-h_{\mu}$ and noting that
\begin{eqnarray*}
 && \frac {\lambda_{k}}{\Lambda_{k+1}}\Big ( -\frac {\Lambda_k}{\lambda_{k}}\log \Big (\frac {\lambda_{k+1}}{\lambda_{k}} \Big) +1-M \Big )+ \frac {\Lambda_{k}}{\Lambda_{k+1}}\sum^{+\infty}_{i=2}\frac 1{i}\Big ( \frac {\lambda_{k}}{\Lambda_{k}}e^{-u} \Big )^i  \\
&< & \frac {\lambda_{k}}{\Lambda_{k+1}}\Big | -\frac {\Lambda_k}{\lambda_{k}}\log \Big (\frac {\lambda_{k+1}}{\lambda_{k}} \Big) +1-M \Big |+\frac {\Lambda_{k}}{\Lambda_{k+1}}\sum^{+\infty}_{i=2}\frac 1{i}\Big ( \frac {\lambda_{k}}{\Lambda_{k}} \Big )^i < \frac 3{4}\frac {\lambda^2_{k}}{\Lambda_{k}\Lambda_{k+1}},
\end{eqnarray*}
  because of $e^{-u}<1$ and \eqref{5.1}.

  Since $\mu < e^M$ and by Lemma \ref{lem5.1}, $M- \epsilon <h_{k_0} \leq h_k < \log \mu$, we have $0< \log \mu -h_k < 2\epsilon$, and therefore we can replace \eqref{5.2} by the linear recurrence relation
\begin{equation}
\label{5.3}
  h_{k+1}-h_k < \frac {\lambda_k}{\Lambda_{k+1}}\Big (\epsilon (\log \mu-h_{k})+\log (\frac {e^M}{\mu})+\frac 3{4} \frac {\lambda_k}{\Lambda_k} \Big ).
\end{equation}
   Putting 
\begin{equation}
\label{5.3'}
   \epsilon(\log \mu-h_{k})+\log (\frac {e^M}{\mu})- \frac 1{4} \frac {\lambda_k}{\Lambda_k} =t_k,
\end{equation}
   so that it follows from \eqref{5.3} that
\begin{equation*}
   t_{k+1} > t_k\Big (1-\frac {\epsilon \lambda_k}{\Lambda_{k+1}} \Big )+(\frac {\lambda_k}{4\Lambda_k}-\frac {\lambda_{k+1}}{4\Lambda_{k+1}}-\frac {\epsilon \lambda^2_k}{\Lambda_{k}\Lambda_{k+1}}).
\end{equation*}
   As we have assumed that $\{ \lambda_k \}^{\infty}_{k=1}$ a non-decreasing sequence, we have
\begin{equation*}
   \frac {\lambda_k}{4\Lambda_k}-\frac {\lambda_{k+1}}{4\Lambda_{k+1}}-\frac {\epsilon \lambda^2_k}{\Lambda_{k}\Lambda_{k+1}} \geq \frac {\lambda_k}{4\Lambda_k}-\frac {\lambda_{k+1}}{4\Lambda_{k+1}}-\frac {\epsilon \lambda_k\lambda_{k+1}}{\Lambda_{k}\Lambda_{k+1}} .
\end{equation*}
  It follows from \eqref{1.6} that the right-hand side expression above is positive if we choose $\epsilon$ small enough and we may assume that our $0< \epsilon<1/2$ is so chosen. Note that this also implies that $0< \log \mu -h_k < 2\epsilon<1$. It follows that
\begin{equation}
\label{5.4}
   t_{k+1} > t_k\Big (1-\frac {\epsilon \lambda_k}{\Lambda_{k+1}} \Big ) \geq t_k\Big (1-\frac {\epsilon \lambda_{k+1}}{\Lambda_{k+1}} \Big ).
\end{equation}
   By Lemma \ref{lem5.1} we have $t_{k_0} >0$ so that the above implies $t_k>0$ for all $k$ under consideration.

   It follows from \eqref{5.4} and $1-\epsilon x > (1-x)^{\epsilon}, 0<x<1$ that
\begin{equation*}
   t_{k+1} > t_k(\Lambda_k)^{\epsilon}(\Lambda_{k+1})^{-\epsilon}
=t_{k}(\frac {\Lambda_{k}}{\lambda_k})^{\epsilon}(\frac {\Lambda_{k+1}}{\lambda_k})^{-\epsilon}.
\end{equation*}
   It follows from \eqref{1.6} that the sequence $\{ \Lambda_k/\lambda_k \}^{\infty}_{k=1}$ is increasing and we deduce that
\begin{equation}
\label{5.6}
   t_{k+1} > t_{k_0}(\frac {\Lambda_{k_0}}{\lambda_{k_0}})^{\epsilon}(\frac {\Lambda_{k+1}}{\lambda_k})^{-\epsilon},
\end{equation}
   for all $k$ under consideration, i.e. for all $k$ for which $h_k < \log \mu$. This is certainly satisfied if $t_k>\log (e^M/\mu)$, and \eqref{5.6} guarantees that this is true as long as the right-hand side expression of \eqref{5.6} is $>\log (e^M/\mu)$. Therefore
\begin{equation}
\label{5.7}
  t_k \geq t_{k_0}(\frac {\Lambda_{k_0}}{\lambda_{k_0}})^{\epsilon}(\frac {\Lambda_{k}}{\lambda_{k-1}})^{-\epsilon}
\end{equation}
   for all $k \geq k_0$ satisfying
\begin{equation}
\label{5.8}
  \frac {\Lambda_{k}}{\lambda_{k-1}} < \Big (\frac {\Lambda_{k_0}}{\lambda_{k_0}} \Big )t^{1/\epsilon}_{k_0}\Big (\log (e^M/\mu) \Big )^{-1/\epsilon},
\end{equation}
   and we are sure that no breakdown occurs in this range.

   Now we return to the discussion on \eqref{5.0} and if $0< h <\log \mu$, we have, on using $e^{-u}>1-u+u^2/3$  for $0<u<1$ and $0 < \log (e^M/\mu) <1$, that
\begin{eqnarray*}
 && e^{h-\log \mu}-h+M-1 \\
&>& \log (e^M/\mu)+\frac 1{3}(\log \mu - h)^2 >  \Big ( \log (e^M/\mu) \Big )^2+\frac 1{3}(\log \mu - h)^2 \\
& >& \frac 1{8}\Big (2\log (e^M/\mu)+\log \mu - h\Big )^2,
\end{eqnarray*} 
  where the last inequality above follows from $u^2+v^2/3 > u^2+(v/2)^2 \geq (u+v/2)^2/2$ for $u, v>0$. Apply this with $h=h_k$ and note that it follows from \eqref{5.3'} and \eqref{5.7} that
\begin{equation*}
   (\epsilon+1)\Big (\log \mu-h_{k}+2\log (\frac {e^M}{\mu}) \Big ) > t_k \geq t_{k_0}(\frac {\Lambda_{k_0}}{\lambda_{k_0}})^{\epsilon}(\frac {\Lambda_{k}}{\lambda_{k-1}})^{-\epsilon},
\end{equation*}
   This implies that the left-hand side of \eqref{5.0} is at least
\begin{equation*}
    \frac {t^2_{k_0}}{8(\epsilon+1)^2}(\frac {\Lambda_{k_0}}{\lambda_{k_0}})^{2\epsilon}(\frac {\Lambda_{k}}{\lambda_{k-1}})^{-2\epsilon}.
\end{equation*}
   This holds for $k$ when \eqref{5.8} is satisfied. It follows from \eqref{1.3} that 
$\lambda_{k}/\lambda_{k-1}$ is bounded above for any $k \geq 2$. Let $c_1$ denote such an upper bound and we conclude that 
the left-hand side of \eqref{5.0} is at least
\begin{equation*}
    \frac {t^2_{k_0}}{8(\epsilon+1)^2c^{2\epsilon}_1}(\frac {\Lambda_{k_0}}{\lambda_{k_0}})^{2\epsilon}(\frac {\Lambda_{k}}{\lambda_{k}})^{-2\epsilon}:=c_2(\frac {\Lambda_{k}}{\lambda_{k}})^{-2\epsilon}.
\end{equation*}
   Other $k$'s do not cause much trouble. First, for the values $1 \leq k < k_0$, we have $h_k(\mu) \leq h_{k_0}(\mu) <  \log \mu- M\lambda_{k_0}/(2\Lambda_{k_0})$ by Lemma \ref{lem5.1} and the fact that $h_k$ increases as $k$ increases. It follows that
\begin{equation*}
 e^{h_k-\log \mu}-h_k+M-1  >  \frac 1{3}(\log \mu - h_k)^2 >  \frac {M^2\lambda^2_{k_0}}{12\Lambda^2_{k_0}}\geq \frac {M^2\lambda^2_{k_0}}{12\Lambda^2_{k_0}}(\frac {\Lambda_{k}}{\lambda_{k}})^{-2\epsilon}:=c_3(\frac {\Lambda_{k}}{\lambda_{k}})^{-2\epsilon}.
\end{equation*} 
   Now, for the remaining case $k_0 \leq k \leq  N_{\mu}$ (which is empty if $\mu=e^M$) such that
\begin{equation*}
   \frac {\Lambda_{k}}{\lambda_{k-1}} \geq \Big (\frac {\Lambda_{k_0}}{\lambda_{k_0}} \Big )t^{1/\epsilon}_{k_0}\Big (\log (e^M/\mu) \Big )^{-1/\epsilon},
\end{equation*}
   we use that 
\begin{equation*}
 e^{h-\log \mu}-h+M-1  > \log (e^M/\mu)
\end{equation*} 
  for all $h$ to see that the left-hand side of \eqref{5.0} is at least
\begin{equation*}
  \Big (\frac {\Lambda_{k_0}}{\lambda_{k_0}} \Big )^{\epsilon}\frac {t_{k_0}}{c^{\epsilon}_1}(\frac {\Lambda_{k}}{\lambda_{k}})^{-2\epsilon}:=c_4(\frac {\Lambda_{k}}{\lambda_{k}})^{-2\epsilon}.
\end{equation*}
  In all three cases the constants are independent of $\mu$ and $k$, so on letting $c=\min (c_2, c_3, c_4)$ and $\delta=2\epsilon$ completes the proof of the lemma.
\end{proof}

\begin{lemma}
\label{lem5.3}
  There exist numbers $\beta$, $e^{M-1}<\beta<e^M$ such that for all $\mu$ satisfying $\beta < \mu < e^M$ there exists an index $N <  N_{\mu}$ with $h_N > \max (2, 2M)$.
\end{lemma}
\begin{proof}
  We apply Lemma \ref{lem5.1} with $\eta$ large enough and some $0< \epsilon <1$, so that the following estimation holds for any integer $k \geq \eta$:
\begin{equation}
\label{5.12}
 \lambda_k/\Lambda_k < e^{-2-\max (2, 2M)}/2, \hspace{0.1in} \Big | -\frac {\Lambda_k}{\lambda_{k}}\log \Big (\frac {\lambda_{k+1}}{\lambda_{k}} \Big) +1-M \Big |< 1-\log 2 ,
\end{equation}  
 and Lemma \ref{lem5.1} provides us with $k_0>\eta$ and $\beta$ such that \eqref{5.00} holds. We now consider the numbers $h_{k_0}, h_{k_0+1}, \ldots$ as far as they are $<\max (2, 2M)+M+1$. If $k \geq k_0$, $h_k<\max (2, 2M)+M+1$, we have
\begin{equation}
\label{5.14}
  \mu^{-1}\lambda_k/\Lambda_ke^{h_k} \leq 1/2,
\end{equation}
   so that by our definition of the breakdown index (see \eqref{3.1}), we have $k < N_{\mu}$. It also follows from \eqref{3.2}-\eqref{4.10}, on using $e^h/\mu-h+M-1 > \log(e^M/\mu)$, that
\begin{equation*}
  h_{k+1}-h_k > \frac {\lambda_k}{\Lambda_{k+1}}\log(e^M/\mu)+O(\frac {\lambda^2_k}{\Lambda_{k}\Lambda_{k+1}}).
\end{equation*}
  The lower bound above shows that not for all $k \geq k_0$ we have $h_k \leq \max (2, 2M)$, since $\sum^{+\infty}_{k_0}(h_{k+1}-h_k)$ would diverge in view of \eqref{3.5'}.
   
   Now, \eqref{5.14},  implies that (with $u=\log \mu-h_{\mu}$ here)
\begin{equation*}
  \frac {\Lambda_{k}}{\Lambda_{k+1}}\sum^{+\infty}_{i=2}\frac 1{i}\Big ( \frac {\lambda_{k}}{\Lambda_{k}}e^{-u} \Big )^i \leq \sum^{+\infty}_{i=2}\frac 1{i}( \frac {1}{2} )^i = \log 2 -1/2.
\end{equation*}
   It follows from this and \eqref{3.2}, \eqref{4.0}, \eqref{5.12}, \eqref{5.14} that 
\begin{equation*}
  h_{k+1}-h_k < \frac {\lambda_k}{\Lambda_{k+1}}(\frac {e^{h_k}}{\mu}-\mu+M-1)+1-\log 2 +\log 2-1/2.
\end{equation*}
   When $M \leq 1$, the above can be estimated by, via \eqref{5.12},
\begin{equation*}
  h_{k+1}-h_k < \frac {\lambda_k}{\Lambda_{k}}\frac {e^{h_k}}{\mu}+1/2<1 <M+1.
\end{equation*}
   Similarly, when $M>1$, we get
\begin{equation*}
  h_{k+1}-h_k < \frac {\lambda_k}{\Lambda_{k}}\frac {e^{h_k}}{\mu}+M-1+1/2< M+1.
\end{equation*}
   It follows from the above that if we let $h_{k_1}$ be the last one below $\max (2, 2M)$, then $h_{k_1+1}$ is still below $\max (2, 2M)+M+1$ so that we can take $N=k_1+1$ here and this completes the proof.
\end{proof}

\section{Proof of Theorem \ref{thm2}}
\label{sec 6} \setcounter{equation}{0}
   As suggested by the discussion in Section \ref{sec 4}, we shall study $\theta(h_k)$, where $\theta$ is defined by
\begin{equation*}
   \theta(y)=\int^{y}_{0}\frac {dx}{e^x/\mu-x+M-1}.
\end{equation*}
   We first simplify the recurrence formula \eqref{3.2}. Assuming
\begin{equation}
\label{6.1}
    e^{M-1} < \mu \leq e^M, \hspace{0.1in} h_k <\max (2, 2M),
\end{equation}
   we may also assume $k$ is large enough so that \eqref{3.1} is not satisfied. We have
\begin{equation*}
   h_{k+1}-h_k=\frac {\lambda_k}{\Lambda_{k+1}}\Big (\frac {e^{h_{k}}}{\mu}-h_k+M-1+\gamma_k  \Big ),
\end{equation*}
   where
\begin{equation*}
  |\gamma_k| \leq \Big |-\frac {\Lambda_k}{\lambda_{k}}\log \Big (\frac {\lambda_{k+1}}{\lambda_{k}}\Big ) -M+1 \Big |+\frac {\Lambda_{k}}{\lambda_{k}}\sum^{+\infty}_{i=2}\frac 1{i}\Big ( \frac {\lambda_{k}}{\Lambda_{k}}e^{h_{\mu}-\log \mu} \Big )^i \leq C_2\frac {\lambda_k}{\Lambda_{k}},
\end{equation*}
   for some constant $C_2>0$. It follows from this that there exists a constant $C_3>0$ such that 
\begin{equation*}
   |h_{k+1}-h_k| \leq C_3\frac {\lambda_k}{\Lambda_{k+1}}.
\end{equation*}
   We then deduce easily from above that for $h_k \leq x \leq h_{k+1}$,
\begin{equation*}
   \Big |\frac {e^{x}}{\mu}-x-(\frac {e^{h_{k}}}{\mu}-h_k) \Big | \leq C_4\frac {\lambda_k}{\Lambda_{k+1}} \leq C_4\frac {\lambda_k}{\Lambda_{k}},
\end{equation*}
   where $C_4>0$ is a constant not depending on $\mu$ or $k$ (still assuming \eqref{6.1}).

   We now apply the mean value theorem to get:
\begin{equation*}
   \theta(h_{k+1})-\theta(h_{k})=(h_{k+1}-h_k)\theta '(x)
\end{equation*}
   with some $x$ in between $h_k$ and $h_{k+1}$. Hence it follows from our discussion above that
\begin{equation*}
   \theta(h_{k+1})-\theta(h_{k})=\frac {\lambda_k}{\Lambda_{k+1}}\frac {H+\gamma_k}{H+\gamma '_k},
\end{equation*}
  where 
\begin{equation*}
  H=\frac {e^{h_{k}}}{\mu}-h_k+M-1, \hspace{0.1in} |\gamma_k| \leq C_2\frac {\lambda_k}{\Lambda_{k}}, \hspace{0.1in} |\gamma '_k| \leq C_4\frac {\lambda_k}{\Lambda_{k}}.
\end{equation*}
   
   We now apply Lemma \ref{lem5.2} to conclude that there exists a $\beta_1$ with $e^{M-1} < \beta_1 < e^M$ and a $c>0$, $0<\delta <1$  such that for all $k$ satisfying $1 \leq k \leq N_{\mu}$, we have
\begin{equation*}
  H > c\Big ( \frac {\Lambda_k}{\lambda_k} \Big )^{-\delta}.
\end{equation*}   
   This implies that
\begin{equation*}
 |\gamma_k| \leq \frac {C_2}{c} \Big (\frac {\lambda_k}{\Lambda_{k}} \Big )^{1-\delta}H, \hspace{0.1in} |\gamma '_k| \leq \frac {C_4}{c} \Big (\frac {\lambda_k}{\Lambda_{k}} \Big )^{1-\delta}H.
\end{equation*} 
   Note it follows from \eqref{3.5'} that
\begin{equation*}
  \frac {\lambda_k}{\Lambda_{k+1}}-\frac {\lambda_k}{\Lambda_{k}}=-\frac {\lambda_k\lambda_{k+1}}{\Lambda_{k}\Lambda_{k+1}}=O(\frac 1{k^2}).
\end{equation*}
   It follows from this and \eqref{3.5'} that we can find an integer $m$, independent of $\mu$ such that for $k>m, h_k < \max (2, 2M)$, we have
\begin{equation*}
   \theta(h_{k+1})-\theta(h_{k})=\frac {\lambda_k}{\Lambda_{k}}\frac {H+\gamma_k}{H+\gamma '_k}+O(\frac 1{k^2})=\frac {C}{k}+O \Big (\frac 1{k^2}+\frac 1{k^{2-\delta}} \Big ).
\end{equation*}
   We recast the above as
\begin{equation*}
   |\theta(h_{k+1})-\theta(h_{k})-C \log (1+1/k) | =O \Big (\frac 1{k^2}+\frac 1{k^{2-\delta}} \Big ).
\end{equation*}
   Now assuming $\mu < e^M$, we take the sum over the values $m \leq k <N$, where $N$ is the first index with $h_N > \max (2, 2M)$ (see Lemma \ref{lem5.3}). This gives us
\begin{equation*}
   |\theta(h_{N})-C \log N | =O (1)+\log m+\theta(h_m).
\end{equation*}

   By Lemma \ref{lem5.1}, for any $\eta>M$, there exists $\beta_2, \beta_1 < \beta_2 < e^M$ and $k_0>\eta$ so that $h_{k_0}(\mu) < \log \mu-M\lambda_{k_0}/(2\Lambda_{k_0})$. We now further take the integer $m$ to be equal to this $k_0$. Thus, the maximum of the integrand in $\theta(h_m)$ is attained at $x=h_m$ and that
\begin{eqnarray*}
 &&  e^{h_m}/\mu-h_m+M-1 \\
 & \geq &  e^{-M\lambda_m/(2\Lambda_m)}-\log \mu+M\lambda_m/(2\Lambda_m)+M-1 \\
 & > &  1-M\lambda_m/(2\Lambda_m)+M^2\lambda^2_m/(8\Lambda^2_m)-\log \mu+M\lambda_m/(2\Lambda_m)+M-1 \\
  &=& M-\log \mu+M^2\lambda^2_m/(8\Lambda^2_m)>M^2\lambda^2_m/(8\Lambda^2_m).
\end{eqnarray*}
  It follows that
\begin{equation*}
  \theta(h_m)=\int^{h_m}_{0}\frac {dx}{e^x/\mu-x+M-1} < (\log \mu )(8\Lambda^2_m)/(M^2\lambda^2_m)=O(1).
\end{equation*}
   We deduce from this that
\begin{equation}
\label{6.2}
   |\theta(h_{N})-C \log N | =O (1).
\end{equation}

   It is not difficult to find the asymptotic behavior of $\theta(\infty)$. If $\mu < e^M$, $\mu \rightarrow e^M$, then routine methods (cf. Sec. \ref{sec 4}) lead to
\begin{equation*}
   \theta(\infty)=\int^{\infty}_{0}\frac {dx}{e^x/\mu-x+M-1}=\sqrt{2}\pi\Big (\log (e^M/\mu ) \Big )^{-1/2}+O(1).
\end{equation*}
   It is also easy to see that $\theta(\infty)-\theta(\max (2, 2M)) = O(1)$. As $h_{N} \geq \max (2, 2M)$, we have $\theta(\max (2, 2M)) \leq \theta(h_{N}) < \theta(\infty)$.  
It follows from \eqref{6.2} that
\begin{equation*}
   \log N=\frac {\sqrt{2}\pi}{C} \Big (\log (e^M/\mu ) \Big )^{-1/2}+O(1).
\end{equation*}
  According to \eqref{3.6} and our discussion in Section \ref{sec 4}, this completes the proof of \eqref{4.2} and it was already shown there that \eqref{4.2} leads to our assertion for Theorem \ref{thm2}.
  
\section{An Application of Theorem \ref{thm2}}
\label{sec 7} \setcounter{equation}{0}
   As an application of Theorem \ref{thm2}, we consider in this section the case $\lambda_k=k^{\alpha}$ for $\alpha \geq 1$. Certainly, the sequence $\{ k^{\alpha} \}^{\infty}_{k=1}$ is a non-decreasing sequence satisfying \eqref{1.3}. We note the following
\begin{lemma}
\label{lem0}
    Let $\alpha \geq 1$ be fixed. For any integer $n \geq 1$, we have
\begin{equation}
\label{7.1}
    \frac {\alpha}{\alpha+1}\frac {n^{\alpha}(n+1)^{\alpha}}{(n+1)^{\alpha}-n^{\alpha}} \leq  \sum^n_{i=1}i^{\alpha} \leq  \frac {(n+1)^{\alpha+1}}{\alpha+1}.
\end{equation}
\end{lemma}
   We point out here the left-hand side inequality above is \cite[Lemma 2, p.18]{L&S} and the right-hand side inequality can be easily shown by induction. 

   It follows readily from the above lemma that \eqref{3.5'} holds with $C=\alpha+1$. We note here it is easy to see that \eqref{5} with $M=1/C$ follows from the left-hand side inequality of \eqref{7.1}, which implies
\begin{equation*}
\label{101}
 \sum_{i=1}^{n+1}i^{\alpha}/(n+1)^{\alpha}-\sum_{i=1}^ni^{\alpha}/n^{\alpha}
 =1+ \Big (\frac 1{(n+1)^{\alpha}}-\frac 1{n^{\alpha}} \Big )\sum^n_{i=1}i^{\alpha} \leq \frac {1}{\alpha+1}.
\end{equation*}
  This combined with the upper bound in \eqref{7.1} also leads to \eqref{1.5} easily. 

  Now, to show \eqref{1.4}, we assume \eqref{1.6} for the moment and note that
\begin{equation*}
  \log \Big(\frac {\Lambda_{k+1}/\lambda_{k+1}}{\Lambda_k/\lambda_k} \Big ) =\log \Big(1+\frac {\Lambda_{k+1}/\lambda_{k+1}-\Lambda_k/\lambda_k}{\Lambda_k/\lambda_k} \Big ) \geq \frac {\Lambda_{k+1}/\lambda_{k+1}-\Lambda_k/\lambda_k}{\Lambda_{k+1}/\lambda_{k+1}}.
\end{equation*}
   We then deduce that \eqref{1.4} follows from
\begin{equation*}
   \Lambda_{k+1}/\lambda_{k+1}-\Lambda_k/\lambda_k \geq M\lambda_k/\lambda_{k+1}.
\end{equation*}
   Note that the above also establishes \eqref{1.6}. In our case, it is easy to see that this becomes (for any $n \geq 1$):
\begin{equation}
\label{7.2}
  \sum^{n}_{i=1}i^{\alpha} \leq \frac {\alpha}{\alpha+1}\frac {(n+1)^{2\alpha}}{(n+2)^{\alpha}-(n+1)^{\alpha}}.
\end{equation}
   To show this, we define
\begin{equation*}
  P_{n}(\alpha)= \left(\frac {1}{n} \sum_{i=1}^{n}i^{\alpha}\bigg/
\frac {1}{n+1}\sum_{i=1}^{n+1}i^{\alpha}\right)^{1/\alpha}.
\end{equation*}
  We recall that Bennett \cite{Be} proved that for $\alpha \geq 1$,
\begin{equation*}
  P_n(\alpha) \leq P_n(1)=\frac
   {n+1}{n+2}.
\end{equation*}
   It is easy to see that this is equivalent to 
\begin{equation*}
   \sum^{n}_{i=1}i^{\alpha} \leq \frac {n(n+1)^{2\alpha}}{(n+1)(n+2)^{\alpha}-n(n+1)^{\alpha}}.
\end{equation*}
   Thus, in order to prove \eqref{7.2}, it suffices to prove the following
\begin{equation*}
   \frac {n(n+1)^{2\alpha}}{(n+1)(n+2)^{\alpha}-n(n+1)^{\alpha}} \leq \frac {\alpha}{\alpha+1}\frac {(n+1)^{2\alpha}}{(n+2)^{\alpha}-(n+1)^{\alpha}}.
\end{equation*}
   The above inequality can be seen easily to be equivalent to the following
\begin{equation*}
  n\Big ((n+2)^{\alpha} -(n+1)^{\alpha} \Big ) \leq \alpha (n+2)^{\alpha},
\end{equation*}
   which follows easily from the mean value theorem. Thus, we have shown, as a consequence of Theorem \eqref{thm2} the following
\begin{cor}
\label{cor1}
  Fix $\alpha \geq 1$ and let $\lambda_k=k^{\alpha}$ for $k \geq 1$. Then inequality \eqref{3} holds with
\begin{equation*}
   U_N=e^{1/(\alpha+1)}- \frac {2\pi^2e^{1/(\alpha+1)}}{(\alpha+1)^2(\log N)^2}+O\Big (\frac 1{(\log N)^3} \Big ).
\end{equation*}  
\end{cor}


\end{document}